\documentclass{amsart}
\usepackage{amsmath}
\usepackage{amsopn}
\usepackage{amssymb}
\usepackage{latexsym}
\usepackage{amsfonts}
\usepackage{amsthm}
\usepackage[all]{xy}

\newtheorem{thm}{Theorem}[section]
\newtheorem{pro}[thm]{Proposition}
\newtheorem{lem}[thm]{Lemma}
\newtheorem{cla}[thm]{Claim}

\newtheorem{cor}[thm]{Corollary}

\theoremstyle{definition}
\newtheorem{obs}[thm]{Observation}

\newtheorem{rem}[thm]{Remark}
\newtheorem{exa}[thm]{Example}
\newtheorem{defn}[thm]{Definition}

\newtheorem{conj}[thm]{Conjecture}

\newcommand{\sumlim}{\sum\limits}
\newcommand{\een}{\end{enumerate}}
\newcommand{\blem}{\begin{lem}}
\newcommand{\elem}{\end{lem}}
\newcommand{\bcl}{\begin{cla}}
\newcommand{\ecl}{\end{cla}}
\newcommand{\ethm}{\end{thm}}
\newcommand{\bpr}{\begin{pro}}
\newcommand{\epr}{\end{pro}}
\newcommand{\bco}{\begin{cor}}
\newcommand{\eco}{\end{cor}}
\newcommand{\bcon}{\begin{conj}}
\newcommand{\econ}{\end{conj}}
\newcommand{\bde}{\begin{defn}}
\newcommand{\ede}{\end{defn}}
\newcommand{\bex}{\begin{exa}}
\newcommand{\eexa}{\end{exa}}
\newcommand{\bobs}{\begin{obs}}
\newcommand{\eobs}{\end{obs}}
\newcommand{\bexe}{\begin{exe}}
\newcommand{\eexe}{\end{exe}}

\begin{document}
\title[Congruence B-Orbits of Symmetric Matrices]{Congruence B-Orbits and the Bruhat Poset of Involutions of the Symmetric Group}
\author{Eli Bagno and Yonah Cherniavsky}

\thanks{The second author was supported by
the Swiss National Science Foundation.}
 \address{Jerusalem College of Technology, Department of Mathematics and Computer Science, Ariel University Center of Samaria, Israel}
 \email{bagnoe@jct.ac.il , chrnvsk@gmail.com}

\begin{abstract} We study the poset of Borel congruence classes of symmetric
matrices ordered by containment of closures. We give a combinatorial description of this poset and calculate its rank function. We discuss the relation between this poset and the Bruhat poset of involutions of the symmetric group.
\end{abstract}
\maketitle
\section{Introduction}

A remarkable property of the Bruhat decomposition of $GL_n(\mathbb C)$ (i.e. the decomposition of $GL_n(\mathbb C)$ into double cosets $\left\{B_1\pi B_2\right\}$ where $\pi\in S_n$ , $B_1,B_2\in\mathbb B_n(\mathbb C)$ -- the subgroup
of upper-triangular invertible matrices ) is that the natural order on double cosets (defined by containment of closures) leads to the same poset as the combinatorially defined Bruhat order on permutations of $S_n$ (for $\pi,\sigma\in S_n$, $\pi\leqslant\sigma$ if $\pi$ is a subword of $\sigma$ with respect to the reduced form in Coxeter generators). L.~Renner introduced and developed the beautiful theory of Bruhat decomposition for not necessarily invertible matrices, see~\cite{R} and~\cite{R1}. When the Borel group acts on all the matrices, the double cosets are in bijection with partial permutations which form a so called {\rm rook monoid} $R_n$ which is the finite monoid whose elements are the 0-1 matrices with at most one nonzero entry in each row and column. The group of invertible elements of $R_n$ is isomorphic to the symmetric group $S_n$. Another efficient, combinatorial description of the Bruhat ordering on $R_n$ and a useful, combinatorial formula for the length function on $R_n$ are given by M.~Can and L.~Renner in~\cite{CR}.

The Bruhat poset of involutions of $S_n$ was first studied by F.~Incitti in~\cite{I} from a purely combinatorial point of view. He proved that this poset is graded, calculated the rank function and also showed several other important properties of this poset. \\

In this paper we present a geometric interpretation of this poset and its natural generalization, considering the action of the Borel subgroup on symmetric matrices by
congruence. Denote by $\mathbb B_n(\mathbb C)$ the Borel
subgroup of $GL_n(\mathbb C)$, i.e. the group of invertible
upper-triangular $n\times n$ matrices over the complex numbers.
Denote by $\mathbb S(n, \mathbb C)$ the set of all complex symmetric
$n\times n$ matrices. The congruence action of
$B\in\mathbb B_n(\mathbb C)$ on $S\in\mathbb S(n, \mathbb C)$ is
defined in the following way:
$$S\,\,\longmapsto\,\,B^tSB\,\,.$$
The orbits of this action (to be precisely correct, we must say $S\,\mapsto\,\left(B^{-1}\right)^tSB^{-1}$ to get indeed a group action) are called  the {\it congruence B-orbits}.
It is known that the orbits of this action may be indexed by partial $S_n$-involutions (i.e. symmetric $n\times n$ matrices with at most one 1 in each row and in each column) (see~\cite{S}). Thus, if $\pi$ is such a partial involution, we denote by
$\mathcal C_\pi$ the corresponding congruence B-orbit of symmetric
matrices. The poset of these orbits gives a natural extension of the Bruhat poset of regular (i.e. not partial) involutions of $S_n$. If we restrict this action to the set of invertible symmetric matrices we get a poset of orbits that is isomorphic to the Bruhat poset of involutions of $S_n$ studied by F.~Incitti. \\

Here, we give another view of the rank function of this poset, combining combinatorics with the geometric nature of it. The rank function equals to the dimension of the orbit variety. We give two combinatorial formulas for the rank function of the poset of partial involutions (Theorems~\ref{PosetRankFunction} and~\ref{genincitti}). The result of Incitti that the Bruhat poset of involutions of $S_n$ is graded and his formula for the rank function of this poset follow from our exposition (Corollary~\ref{incittigrfr}).

At the end of the paper we briefly discuss how our view of the rank function can be applied to the non-symmetric case, i.e. how to find  the rank function of the Bruhat poset of all (not necessarily symmetric) partial permutations in a similar way.

\section{Preliminaries}
\subsection{Permutations and partial permutations. The Bruhat order}
The {\rm Bruhat order} on permutations of $S_n$ is defined as follows: $\pi\leqslant\sigma$ if $\pi$ is is a subword of $\sigma$ in Coxeter generators $s_1=(1,2)$, $s_2=(2,3)$,...,$s_{n-1}=(n-1,n)$. It it well studied from various points of view. The {\rm rank function} is the length in Coxeter generators which is exactly the number of inversions in a permutation. A {\rm permutation matrix} is a square matrix which has exactly one 1 in each row and each column while all other entries are zeros. A {\rm partial permutation} is an injective map defined on a subset of $\{1,2,..,n\}$. A {\rm partial permutation matrix} is a square matrix which has at most one 1 at each row and each column and all other entries are zeros. So, if we delete the zero rows and columns from a partial permutation matrix we get a (regular) permutation matrix of smaller size, we will use this view later. See works of L. Renner \cite{R1} and \cite{R} where the Bruhat order on partial permutations is introduced and studied.

\subsection{Partial order on orbits}

\label{OrbOrder}
When an algebraic group acts on a set of matrices, the classical
partial order on the set of all orbits is defined as follows:
$$\mathcal{O}_1 \leq  \mathcal{O}_2\,\,\Longleftrightarrow\,\,\mathcal{O}_1 \subseteq \overline{\mathcal{O}_2}$$
where $\overline{S}$ is the  (Zarisski) closure
of the set $S$.

\begin{rem}
Note that $\mathcal O_1\subseteq\overline{\mathcal
O_2}\Longrightarrow\overline{\mathcal
O_1}\subseteq\overline{\mathcal O_2}$ for any two sets $\mathcal
O_1,\mathcal O_2$.
\end{rem}

\section{Rank-control matrices}

In this section we define the {\it rank control matrix} which will
    turn out to be a key corner in the identification of our poset. We
start with the following definition:

\bde
 Let $X=\left(x_{ij}\right)$ be an
$n\times m$ matrix. For each $1 \leq k \leq n$ and $1 \leq l \leq m$, denote by $X_{k\ell}$ the
upper-left $k\times\ell$ submatrix of $X$. We denote by $R(X)$
the $n\times m$ matrix whose entries
are: $r_{k\ell}=rank\left(X_{k\ell}\right)$ and call it the {\it rank control matrix} of $X$.\ede

It follows from the definitions that for each matrix $X$, the entries of $R(X)$ are
nonnegative integers which do not decrease in rows and columns and
each entry is not greater than its row and column number. If $X$
is symmetric, then $R(X)$ is symmetric as well.

\begin{exa}
$$I_3=\begin{pmatrix}
1 &0 &0\\
0 &1 &0\\
0 &0 &1
\end{pmatrix}\quad,\quad R(I_3)= \begin{pmatrix}
1 &1 &1\\
1 &2 &2\\
1 &2 &3
\end{pmatrix}
.$$
\end{exa}

\begin{rem} This rank-control matrix is similar to the one
introduced by A.~Melnikov \cite{M} when she studied the poset (with
respect to the covering relation given in Definition~\ref{OrbOrder})
of adjoint B-orbits of certain nilpotent strictly upper-triangular
matrices.

The rank control matrix is connected also to the work of Incitti
\cite{I} where regular involutions of $S_n$ are
discussed.
\end{rem}

\bpr\label{ulr} Let $X,Y\in GL_n(\mathbb F)$ be such that $Y=LXB$ for
some invertible lower-triangular matrix $L$ and some Borel (i.e.
invertible upper-triangular) matrix $B$. Denote by $X_{k\ell}$ and
$Y_{k\ell}$ the upper-left $k\times\ell$ submatrices of $X$ and $Y$
respectively. Then for all $1\leqslant k,\ell\leqslant n$
$$rank\left(X_{k\ell}\right)=rank\left(Y_{k\ell}\right)\,.$$
 \epr
\begin{proof}
$$\begin{pmatrix}
L_{kk} &0_{k\times(n-k)}\\
* &*
\end{pmatrix}\begin{pmatrix}
X_{k\ell} &*\\
* &*
\end{pmatrix}\begin{pmatrix}
B_{\ell\ell} &*\\
0_{(n-\ell)\times\ell} &*\end{pmatrix}=\begin{pmatrix}
L_{kk}X_{k\ell}B_{\ell\ell} &*\\
* &*
\end{pmatrix}\,,
$$
and therefore, $Y_{k\ell}=L_{kk}X_{k\ell}B_{\ell\ell}$. The Matrices
$L_{kk} $ and $B_{\ell\ell}$ are invertible, which implies that
$Y_{k\ell}$ and $X_{k\ell}$ have equal ranks.
\end{proof}

The rank control matrices of two permutations can be used to compare
between them in the sense of Bruhat order. This is the reasoning for
the next definition:

 \bde\label{rankcontrolorder} Define the following order on
$n \times m$ matrices with positive integer entries: Let
$P=\left(p_{ij}\right)$ and $Q=\left(q_{ij}\right)$ be two such
matrices.

 Then
$$P\leqslant_{\mathcal R}
Q\,\,\Longleftrightarrow\,\,p_{ij}\leqslant q_{ij}\,\,\textrm{for
all}\,\,i,j\,.$$ \ede

The following lemma appears in another form as Theorem 2.1.5 of
\cite{BB}.

 \blem\label{BrOrdPerm} Denote by $\leqslant_B$
the Bruhat order of $S_n$ and let $\pi,\sigma\in S_n$. Then
$$\pi\leqslant_{\mathcal B}\sigma\quad\Longleftrightarrow\quad
R(\pi)\geqslant_{\mathcal R} R(\sigma)\,.$$ In other words, the
Bruhat order on permutations corresponds to the inverse order of
their rank-control matrices. \qed\elem

\section{Partial permutations, Partial Involutions and Congruence B-Orbits}

\bde\label{PartPerm} A {\rm partial permutation} is an $n\times n$
$(0,1)$-matrix such that each row and each column contains at most
one `1'.  \ede

\bde\label{PartInvol} If a partial permutation matrix is symmetric,
then we call it a {\it partial involution}. \ede

The following easily verified lemma claims that partial permutations
are completely characterized by their rank control matrices.

\blem\label{PermRankComplDef}
For two $n\times n$ partial permutation matrices $\pi,\sigma$ we
have
$$R(\pi)=R(\sigma)\iff\pi=\sigma.$$
 \elem

\begin{rem}
Let $U_n$ be the $n \times n$ upper-triangular matrix with '1's on
the main diagonal and in all upper triangle
and let $\pi$ be any partial permutation. Then
$$R(\pi)=U^t\pi U\,\,.$$
\end{rem}
\begin{thm}\label{OrbParam}
There exists a bijection between the set of congruence B-orbits of
symmetric matrices over $\mathbb C$ and the set of partial
involutions.
\end{thm}
\begin{proof} The proof can be obtained by performing a symmetric version of Gauss elimination process. See Theorem 3.2 in \cite{S} for more details.
\end{proof}
\begin{rem}
Note that the elimination process described above works only over
$\mathbb{C}$. In the case of $\mathbb{R}$, the diagonal entries of a
partial involution matrix (which are its fixed points) belong to
$\{0,1-1\}$.

\end{rem}

\section{The Poset of Congruence B-Orbits of Symmetric Matrices}


Here is a direct consequence of Lemma~\ref{PermRankComplDef} and Proposition~\ref{ulr}.
\bpr \label{BrCsRankCont}  All the matrices of a fixed
congruence B-Orbit share a comon rank-control
matrix. In other words, if $\pi$ is a partial $S_n$-involution, and
$C_{\pi}$ is the congruence B-orbit of symmetric matrices associated
with $\pi$ then
$$
\mathcal C_\pi=\left\{S\in \mathbb S(n,\mathbb
C)\,|\,R(S)=R(\pi)\right\}.
$$
\epr

The following lemma describes the orbits:

\blem\label{CosetClosure} Let $\pi$ be a partial involution and let
$R(\pi)$ be its rank-control matrix.  Then
$$
\overline{\mathcal C_\pi}=\left\{S\in\mathbb S(n,\mathbb
C)\,\,|\,\,R(S)\leqslant_{\mathcal R} R(\pi)\right\}\,.
$$
 \elem
\begin{proof}
This lemma follows from Theorem 15.31 of \cite{MS}. Their exposition differs somewhat from ours as it deals with rectangular,
not necessarily symmetric matrices but the differences can be easily overwhelmed by considering also equations of the form $a_{ij}=a_{ji}$ which are polynomial equations with regard to the entries of a matrix.
\end{proof}

\begin{rem} Over the fields $\mathbb C$ and $\mathbb R$
the closure in Lemma~\ref{CosetClosure} may also be considered with respect to the  metric topology.
\end{rem}
The next corollary follows from Lemma~\ref{CosetClosure} and characterizes the order relation of the poset of B-orbits.
\begin{cor}\label{main} Let $\pi$ and $\sigma$ be partial $S_n$-involutions.
Then
$$
\mathcal C_\pi\leqslant_{\mathcal O}\mathcal C_\sigma\iff R(\pi)
\leqslant_{\mathcal R}R(\sigma)
$$
\end{cor}

\section{An example}
In this section we give an example for the poset of B-congruence orbits. We represent each orbit by its
partial involution (see Theorem~\ref{OrbParam}) and write the
rank-control matrix together with each partial involution.
\begin{exa}
This example illustrates the case $n=3$.
$$
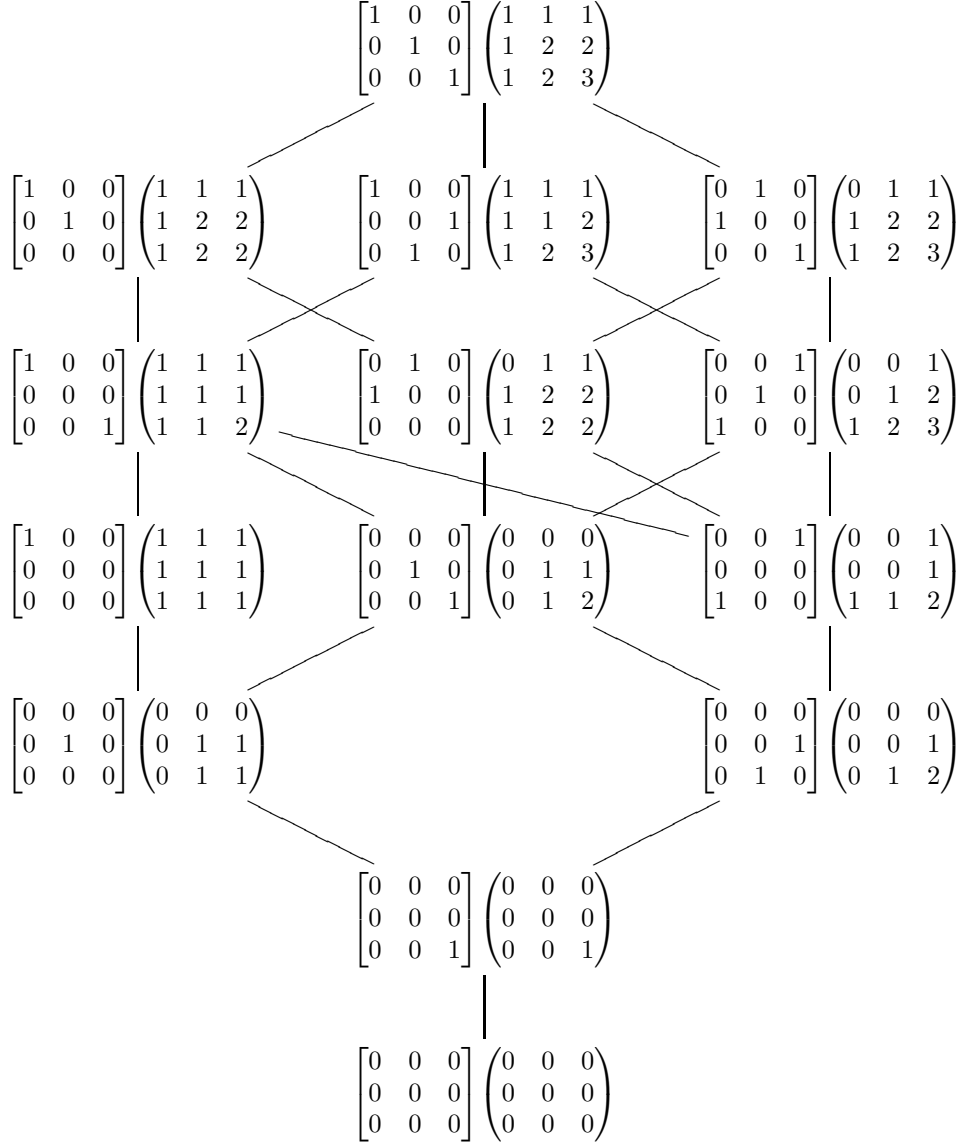
\xymatrix{& {\left[\begin{matrix}
1 &0 &0\\
0 &1 &0\\
0 &0 &1 \end{matrix}\right]}{\left(\begin{matrix}
1 &1 &1\\
1 &2 &2\\
1 &2 &3 \end{matrix}\right)}& \\
{\left[\begin{matrix}
1 &0 &0\\
0 &1 &0\\
0 &0 &0 \end{matrix}\right]}{\left(\begin{matrix}
1 &1 &1\\
1 &2 &2\\
1 &2 &2 \end{matrix}\right)}\ar@{-}[ur] & {\left[\begin{matrix}
1 &0 &0\\
0 &0 &1\\
0 &1 &0 \end{matrix}\right]}{\left(\begin{matrix}
1 &1 &1\\
1 &1 &2\\
1 &2 &3 \end{matrix}\right)} \ar@{-}[u] &{\left[\begin{matrix}
0 &1 &0\\
1 &0 &0\\
0 &0 &1 \end{matrix}\right]}{\left(\begin{matrix}
0 &1 &1\\
1 &2 &2\\
1 &2 &3 \end{matrix}\right)} \ar@{-}[ul]\\
{\left[\begin{matrix}
1 &0 &0\\
0 &0 &0\\
0 &0 &1 \end{matrix}\right]}{\left(\begin{matrix}
1 &1 &1\\
1 &1 &1\\
1 &1 &2 \end{matrix}\right)}\ar@{-}[u]\ar@{-}[ur] &
{\left[\begin{matrix}
0 &1 &0\\
1 &0 &0\\
0 &0 &0 \end{matrix}\right]}{\left(\begin{matrix}
0 &1 &1\\
1 &2 &2\\
1 &2 &2 \end{matrix}\right)}\ar@{-}[ur]\ar@{-}[ul] &
{\left[\begin{matrix}
0 &0 &1\\
0 &1 &0\\
1 &0 &0 \end{matrix}\right]}{\left(\begin{matrix}
0 &0 &1\\
0 &1 &2\\
1 &2 &3 \end{matrix}\right)}\ar@{-}[ul]\ar@{-}[u]\\
{\left[\begin{matrix}
1 &0 &0\\
0 &0 &0\\
0 &0 &0 \end{matrix}\right]}{\left(\begin{matrix}
1 &1 &1\\
1 &1 &1\\
1 &1 &1 \end{matrix}\right)}\ar@{-}[u] & {\left[\begin{matrix}
0 &0 &0\\
0 &1 &0\\
0 &0 &1 \end{matrix}\right]}{\left(\begin{matrix}
0 &0 &0\\
0 &1 &1\\
0 &1 &2 \end{matrix}\right)}\ar@{-}[ur]\ar@{-}[u]\ar@{-}[ul]&
{\left[\begin{matrix}
0 &0 &1\\
0 &0 &0\\
1 &0 &0 \end{matrix}\right]}{\left(\begin{matrix}
0 &0 &1\\
0 &0 &1\\
1 &1 &2 \end{matrix}\right)}\ar@{-}[ul]\ar@{-}[ull]\ar@{-}[u]\\
{\left[\begin{matrix}
0 &0 &0\\
0 &1 &0\\
0 &0 &0 \end{matrix}\right]}{\left(\begin{matrix}
0 &0 &0\\
0 &1 &1\\
0 &1 &1 \end{matrix}\right)}\ar@{-}[u] \ar@{-}[ur]& &
{\left[\begin{matrix}
0 &0 &0\\
0 &0 &1\\
0 &1 &0 \end{matrix}\right]}{\left(\begin{matrix}
0 &0 &0\\
0 &0 &1\\
0 &1 &2 \end{matrix}\right)}\ar@{-}[u]\ar@{-}[ul]\\
& {\left[\begin{matrix}
0 &0 &0\\
0 &0 &0\\
0 &0 &1 \end{matrix}\right]}{\left(\begin{matrix}
0 &0 &0\\
0 &0 &0\\
0 &0 &1 \end{matrix}\right)}\ar@{-}[ul]\ar@{-}[ur]\\
 & {\left[\begin{matrix}
0 &0 &0\\
0 &0 &0\\
0 &0 &0 \end{matrix}\right]}{\left(\begin{matrix}
0 &0 &0\\
0 &0 &0\\
0 &0 &0 \end{matrix}\right)}\ar@{-}[u]}
$$

\end{exa}

\section{The Rank Function}
\bde A poset $P$ is called {\it graded} (or {\it ranked}) if for every $x,y \in P$, any two maximal chains from $x$ to $y$ have the same length.
\ede
\bpr\label{rankdimrenner} The poset of congruence B-orbits of symmetric matrices (with respect to the order $\leqslant_{\mathcal O}$) is a graded poset with the rank function given by the dimension of the closure.
\epr
This proposition is a particular case of the following fact. Let $G$ be a connected, solvable group acting on an irreducible, affine variety $X$. Suppose that there are a finite number of orbits. Let $O$ be the
set of $G$-orbits on $X$. For $x,y\in O$ define $x\leqslant y$ if $x\subseteq\overline y$. Then $O$ is a graded poset.

This fact is given as an exercise in~\cite{R} (exercise 12, page 151) and can be proved using the proof of the theorem appearing of Section~8 of~\cite{R1}. (Note that in our case the Borel group is solvable, the variety of all symmetric matrices is irreducible as a vector space and the number of orbits is finite since there are only finitely many partial permutation.)

A natural problem is to find an
algorithm which calculates the $\dim\overline{\mathcal C_\pi}$
from a partial permutation matrix $\pi$ or from its rank-control
matrix $R(\pi)$. Here we present such an algorithm.

\bde\label{eqcount} Let $\pi$ be a partial permutation matrix and let $R(\pi)=(r_{ij})$ be its rank-control matrix.  Add an extra
$0$ row to $R(\pi)$, pushed one place to the left, i.e. assume that
$r_{0k}=0$ for each $0\leqslant k < n$.

 Denote
$$
\mathfrak D(\pi)=\#\left\{(i,j)\,|\,1\leqslant i\leqslant
j\leqslant n\quad\textrm{and}\quad r_{ij}=r_{i-1,j-1}\right\}.
$$\ede
There are two examples after the proof of Theorem~\ref{PosetRankFunction} (Examples~\ref{exd1} and~\ref{exd2}) and one more example in the end of the paper (Example~\ref{ex3}) of calculation of the parameter $\mathfrak D(\pi)$.
\begin{thm}\label{PosetRankFunction} Let $\pi$ be a partial $S_n$-involution. Then
$$
\dim\,\overline{\mathcal C_\pi}=\frac{n^2+n}{2}-\mathfrak D(\pi).
$$
\end{thm}
\begin{proof}
In this proof we use the notion of variety which corresponds to a fragment of a matrix, we put empty boxes instead of entries that we "cut" from a matrix. Any variety of $n\times n$ matrices is a (Zariski) closed subset of the vector space $\mathbb C^{n^2}$ and the variety which corresponds to a fragment of $n\times n$ matrix is the closure of the projection of the big variety on the corresponding subspace of $\mathbb C^{n^2}$. (Here the projection is the mapping $p:\mathbb C^n\rightarrow\mathbb C^{n-1}$ which "forgets" the last coordinate: $p\left(x_1,x_2,...,x_{n-1},x_n\right)=\left(x_1,x_2,...,x_{n-1}\right)$.) Denote by $V^{kn}$ the variety which is a projection of the certain variety $\overline{\mathcal C_\pi}$ of symmetric $n\times n$ matrices and corresponds to the fragment $\left[\begin{matrix}
a_{11} &a_{12} &\cdots &a_{1,k} &\cdots &\cdots &a_{1,n-1} &a_{1,n} \\
a_{12} &a_{22} &\cdots &a_{2,k} &\cdots &\cdots &a_{2,n-1} &a_{2,n} \\
\cdots &\cdots &\cdots&\cdots &\cdots &\cdots &\cdots &\cdots\\
a_{1,k-1} &a_{2,k-1} &\cdots &a_{k-1,k}&\cdots &\cdots&a_{k-1,n-1} &a_{k-1,n}\\
a_{1,k} &a_{2,k} &\cdots &a_{k,k}&\cdots &\cdots&a_{k,n-1}  &a_{k,n}\\
a_{1,k+1} &a_{2,k+1} &\cdots&a_{k,k+1} &\cdots &\cdots &a_{k+1,n-1} &\square\\
\cdots &\cdots &\cdots &\cdots &\cdots&\cdots &\cdots&\square\\
a_{1,n-1} &a_{2,n-1} &\cdots &\cdots &\cdots&\cdots&a_{n-1,n-1} &\square\\
a_{1,n} &a_{2,n} &\cdots&a_{k,n}&\square&\square&\square&\square
\end{matrix}\right]$. (For $V^{kn}$ the last non empty entry in the $n$-th column is in the row number $k$, all further positions in the $n$-th row and column are empty.) Consider also the variety $V^{k-1,n}$ which corresponds to $\left[\begin{matrix}
a_{11}  &\cdots &a_{1,k-1}  &\cdots &a_{1,n-1} &a_{1,n} \\
\cdots  &\cdots&\cdots  &\cdots &\cdots &\cdots\\
a_{1,k-1}  &\cdots &a_{k-1,k-1} &\cdots&a_{k-1,n-1} &a_{k-1,n}\\
a_{1,k}  &\cdots &a_{k-1,k} &\cdots&a_{k,n-1}  &\square\\
\cdots  &\cdots &\cdots &\cdots &\cdots&\square\\
a_{1,n-1}  &\cdots &\cdots  &\cdots&a_{n-1,n-1} &\square\\
a_{1,n}  &\cdots&a_{k-1,n} &\square&\square&\square
\end{matrix}\right]$. Note that since $V^{kn}$ and $ V^{k-1,n}$ are projections of the same variety $\overline{\mathcal C_\pi}$ and $V^{kn}$ has one more coordinate than $ V^{k-1,n}$, there are only two possibilities for their dimensions:  $\dim V^{kn}=\dim V^{k-1,n}$ or $\dim V^{kn}=\dim V^{k-1,n}+1$.

Now, let us start the course of the proof, by induction on $n$. For $n=1$ the statement is obviously true.

Consider an $n\times n$
partial $S_n$ involution $\pi_n$ and its rank-control matrix $R(\pi_{n})$.  Its upper-left $(n-1)\times (n-1)$
submatrix is the rank-control matrix $R(\pi_{n-1})$ of the
partial $S_{n-1}$-involution $\pi_{n-1}$ which is the upper-left
$(n-1)\times (n-1)$ submatrix of the partial $S_n$-involution
$\pi$. By the induction hypothesis, $
\dim\,\overline{\mathcal C_{\pi_{n-1}}}=\frac{n^2-n}{2}-\mathfrak D\left(\pi_{n-1}\right)$. Now we add the $n$-th column to the partial involution matrix (the $n$-th row is the same and does not provide any new information  because we deal with the symmetric matrices) and consider the $n$-th column of $R(\pi)$. (We also add the $n$-th row but since our matrices are symmetric it suffices to understand only what happens to the dimension when we add the $n$-th column.) We added $n$ new coordinates to the variety $\overline{\mathcal C_{\pi_{n-1}}}$ and we have to show that
$$
\dim\,\overline{\mathcal C_{\pi}}=\dim\,\overline{\mathcal C_{\pi_{n-1}}}+n-\#\left\{(i,n)\,|\,1\leqslant i\leqslant n\quad\textrm{and}\quad r_{in}=r_{i-1,n-1}\right\}\,,\eqno(*)
$$ i.e. not all the $n$ coordinates that we added make the dimension greater but only those of them  for which there is an inequality in the corresponding place of the certain diagonal of the rank-control matrix. The equality $(*)$ implies the statement of our theorem since $\frac{n^2-n}{2}+n=\frac{n^2+n}{2}$ and
$$
\mathfrak D\left(\pi\right)=\mathfrak D\left(\pi_{n-1}\right)+\#\left\{(i,n)\,|\,1\leqslant i\leqslant n\quad\textrm{and}\quad r_{in}=r_{i-1,n-1}\right\}\,.
$$
Obviously, if $r_{1,n}=0$, then $a_{1,n}=0$ for any
$A=\left(a_{ij}\right)^n_{i,j=1}\in\overline{\mathcal C_\pi}$, this
itself is a polynomial equation which makes the dimension lower by
1, while if $r_{1,n}=1$ it means that the rank of the first row is
maximal and therefore, no equation. (In other words,  the dimension of the variety $V^{1n}$ which corresponds to $\left[\begin{matrix}
a_{11} &a_{12} &\cdots &a_{1,n-1} &a_{1,n}\\
a_{12} &a_{22} &\cdots &a_{2,n-1} &\square\\
\cdots &\cdots &\cdots &\cdots &\square\\
a_{1,n-1} &a_{2.n-1} &\cdots &a_{n-1,n-1}  &\square\\
a_{1,n} &\square &\square &\square &\square
\end{matrix}\right]$ is greater by one than the of the variety $V^{0n}$ which corresponds to $\left[\begin{matrix}
a_{11} &a_{12} &\cdots &a_{1,n-1}\\
a_{21} &a_{22} &\cdots &a_{2,n-1} \\
\cdots &\cdots &\cdots &\cdots \\
a_{n-2,1} &a_{n-2,2} &\cdots &a_{n-2,n-1} \\
a_{n-1,1} &a_{n-1,2} &\cdots &a_{n-1,n-1}
\end{matrix}\right]$ when $r_{1,n}=1$ and they have equal dimensions when $r_{1,n}=0$.) Now move down along the $n$-th
column of $R(\pi)$. Again by induction, this time the induction is
on the number of row $k$, assume that for each $1\leqslant i\leqslant k-1$ the dimensions of $V^{in}$ and $V^{i-1,n}$ ar equal iff
$r_{i-1,n-1}=r_{i,n}$ and  $\dim V^{in}=\dim V^{i-1,n}+1$ iff $r_{i-1,n-1}<r_{i,n}$.
First, let $r_{k-1,n-1}=r_{k,n}=c$. Consider a matrix
$A=\left(a_{ij}\right)_{i,j=1}^n\in\overline{\mathcal C_\pi}$ and
consider its upper-left $(k-1)\times(n-1)$ submatrix
$\left[\begin{matrix}
a_{11} &a_{12} &\cdots &a_{1,n-1}\\
a_{21} &a_{22} &\cdots &a_{2,n-1}\\
\cdots &\cdots &\cdots &\cdots \\
a_{k-1,1} &a_{k-1,2} &\cdots &a_{k-1,n-1}
\end{matrix}\right]$. Using the notation introduced in
Proposition~\ref{ulr}, we denote this submatrix as $A_{k-1,n-1}$.  If $c=0$, then $rank A_{kn}=0$, so $A_{kn}$ is a zero matrix and $\dim V^{in}=\dim V^{i-1,n}=0$. Let $c\neq0$.
Since $ rank\left(A_{k-1,n-1}\right)=c$, we can take $c$ linearly independent columns
$\left[\begin{matrix}a_{1,j_1}\\a_{2,j_1}\\ \cdots\\
a_{k-1,j_{1}}\end{matrix}\right]$ , ... , $\left[\begin{matrix}a_{1,j_c}\\a_{2,j_c}\\ \cdots\\
a_{k-1,j_c}\end{matrix}\right]$ which span its column space. Now take only linearly independent rows of the $(k-1)\times c$ matrix $\left[\begin{matrix}
a_{1,j_1} &\cdots &a_{1,j_c} \\
a_{2,j_1} &\cdots &a_{2,j_c} \\
\cdots &\cdots &\cdots\\
a_{k-1,j_1} &\cdots &a_{k-1,j_c}\end{matrix}\right]$ to get a nonsingular $c\times c$ matrix $\left[\begin{matrix}
a_{i_1,j_1} &\cdots &a_{i_1,j_c} \\
a_{i_2,j_1} &\cdots &a_{i_2,j_c} \\
\cdots &\cdots &\cdots\\
a_{i_c,j_1} &\cdots &a_{i_c,j_c}\end{matrix}\right]$. The
equality $r_{k-1,n-1}=r_{k,n}=c\leqslant k-1$ implies that any $(c+1)\times(c+1)$ minor of the matrix $A_{kn}$ is zero, in particular
$\det\left[\begin{matrix}
a_{i_1,j_1} &\cdots &a_{i_1,j_c} &a_{i_1,n}\\
a_{i_2,j_1} &\cdots &a_{i_2,j_c} &a_{i_2,n}\\
\cdots &\cdots &\cdots &\cdots\\
a_{i_c,j_1} &\cdots &a_{i_c,j_c} &a_{i_c,n}\\
a_{k,j_1} &\cdots &a_{k,j_c} &a_{k,n}
\end{matrix}\right]=0$, which is a polynomial equation. This equation is algebraically independent of the similar equations obtained for $1\leqslant i\leqslant k-1$ because it involves the "new" variable -- the entry $a_{k,n}$. It indeed
 involves the entry $a_{k,n}$ since $\det\left[\begin{matrix}
a_{i_1,j_1} &\cdots &a_{i_1,j_c} \\
a_{i_2,j_1} &\cdots &a_{i_2,j_c} \\
\cdots &\cdots &\cdots \\
a_{i_c,j_1} &\cdots &a_{i_c,j_c}
\end{matrix}\right]\neq 0$. This equation means that the variable $a_{k,n}$ is not independent of the coordinates of the variety $V^{k-1,n}$, and therefore $\dim V^{k-1,n}=\dim V^{kn}$.

Now  let $r_{k-1,n-1}<r_{k,n}=c$, and we have to show that in this
case the variable $a_{nk}$ is independent of the coordinates of $V^{k-1,n}$, in other words, we have to show that there is no new equation. Consider the fragment
$\left[\begin{matrix}
r_{k-1,n-1} &r_{k-1,n}\\
r_{k-1,n} &r_{k,n}
\end{matrix}\right]$. There are four possible cases:
\begin{align*}
\left[\begin{matrix}
r_{k-1,n-1} &r_{k-1,n}\\
r_{k-1,n} &r_{k,n}
\end{matrix}\right]=&
\left[\begin{matrix}
c-1 &c-1\\
c-1  &c
\end{matrix}\right]\quad\textrm{or}\quad
\left[\begin{matrix}
c-2 &c-1\\
c-1  &c
\end{matrix}\right]
\quad\textrm{or}\\
& \left[\begin{matrix}
c-1 &c\\
c-1  &c
\end{matrix}\right]\quad\textrm{or}\quad
\left[\begin{matrix}
c-1 &c-1\\
c  &c
\end{matrix}\right]\,\,.
\end{align*}
The equality $r_{k,n}=c$ implies that each $(c+1)\times(c+1)$ minor  of $A_{kn}$ is equal to zero, but we shall see that each such equation is not new, i.e. it is implied by  the equality $r_{k,n-1}=c-1$ or by the equality $r_{k-1,n}=c-1$. In the first three of above four cases we decompose the $(c+1)\times(c+1)$ determinant $\det\left[\begin{matrix}
\cdots &\cdots   \\
\cdots  &a_{k,n}
\end{matrix}\right]$ using the last column. Since in all these cases $r_{k,n-1}=c-1$, each $c\times c$ minor of this decomposition (i.e. each $c\times c$ minor of $A_{k,n-1}$) is zero and therefore, this determinant is zero. In the fourth case we get the same if we decompose the determinant using its last row instead of the last column: since $r_{k-1,n}=c-1$, all the $c\times c$ minors of this decomposition (i.e. all $c\times c$ minor of $A_{k-1,n}$) are zeros and thus, our $(c+1)\times(c+1)$ determinant equals to zero. So there is no algebraic dependence between $a_{kn}$ and the coordinates of $V^{k-1,n}$. Therefore, $\dim V^{kn}=\dim V^{k-1,n}+1$. The case $k=n$ is the same as other cases when $k\leqslant n-1$.
  The proof is completed.
\end{proof}

We end this section with two examples:

\begin{exa}\label{exd1}
 If $\pi=Id$ then clearly $\mathfrak D(\pi)=0$ and $\dim\,\overline{\mathcal C_\pi}=\frac{n^2+n}{2}$. Indeed, $\overline{\mathcal C_{Id}}$ is the variety of all symmetric $n\times n$ matrices and thus its dimension is equal to $n^2$ minus the number of equations of the type $a_{ij}=a_{ji}$, i.e. $\dim\,\overline{\mathcal C_{Id}}=n^2-\frac{n^2-n}{2}=\frac{n^2+n}{2}$.
\end{exa}
\begin{exa}\label{exd2}
 Taking $\pi=\left[\begin{matrix}0 &1 &0\\
1 &0 &0\\
0 &0 &1 \end{matrix}\right]$ with $R(\pi)=\left(\begin{matrix}
0 &1 &1\\
1 &2 &2\\
1 &2 &3 \end{matrix}\right)$ we have diagonals (with added zeros)
$(0,0,2,3)$, $(0,1,2)$ and $(0,1)$. We see that here we have one
place in the beginning of the main diagonal where
$r_{11}=r_{00}=0$ while in all other places $r_{ij}$ is strictly
greater than $r_{i-1,j-1}$. Therefore, $\mathfrak D(\pi)=1$ and $\dim\,\overline{\mathcal C_\pi}=\frac{3^2+3}{2}-1=5$. Indeed,
$$\overline{\mathcal C_\pi}=\left\{(a_{ij})^3_{i,j=1}\,\,|\,\,a_{12}=a_{21},a_{13}=a_{31},a_{23}=a_{32}\,\,\textrm{and}\,\,\,a_{11}=0\right\}.$$ The dimension of the vector space of all $3\times3$ matrices is $3^2=9$ and here we have four algebraically independent equations, so the dimension is 5.
\end{exa}

\section{Another characterization of the parameter $\mathfrak D(\pi)$}
\noindent
Obviously, an $n\times n$ partial involution matrix $\pi$ can be described
uniquely by the pair $\left(\tilde{\pi},\left\{i_1,...,i_k\right\}\right)$, where $n-k$
is the rank of the matrix $\pi$, $\tilde{\pi}\in S_{n-k}$ such
that $\tilde{\pi}^2=Id$ is the regular (not partial) involution of
the symmetric group $S_{n-k}$ and the integers $i_1,...,i_k$ are
the numbers of zero rows (columns) in the matrix $\pi$.

The following theorem is a generalization of the formula for the rank function of the Bruhat poset of the involutions of $S_n$ given by Incitti in~\cite{I}. It is indeed the rank function because we already know that the rank function is the dimension (Proposition~\ref{rankdimrenner}) and the dimension is determined by the parameter $\mathfrak D$ (Theorem~\ref{PosetRankFunction}).
\begin{thm}\label{genincitti}
Following Incitti, denote by $Invol(G)$ the set of all involutions
in the group $G$. Then for a partial permutation
$\pi=\left(\tilde{\pi},\left\{i_1,...,i_k\right\}\right)$, where
$\tilde{\pi}\in Invol(S_{n-k})$ and the integers $i_1,...,i_k$ are
the numbers of zero rows (columns) in the matrix $\pi$ is: $$
\mathfrak
D(\pi)=\frac{exc(\tilde{\pi})+inv(\tilde{\pi})}{2}+\sum_{t=1}^{k}(n+1-i_t)
$$

In other words, $\mathfrak D(\pi)$ equals to the length of
$\tilde{\pi}$ in the poset of the involutions of the group $S_{n-k}$
plus the sum of the numbers of zero rows of the matrix $\pi$, where
the numbers are taken in the opposite order, i.e. the $n$-th row is
labeled by 1, the $(n-1)$-th row is labeled by 2,..., the first row
is labeled by $n$.

\end{thm}

\begin{proof}

We prove by induction on $n$. The claim is trivial for $n=1$ so
assume that it is true for $n-1$ and let $\pi_n$ be a partial
permutation of order $n$. Let $\pi_{n-1}$ be the sub matrix of
$\pi_n$ consisting of the first $n-1$ rows and columns. We introduce some notations:

\begin{itemize}
\item  Let $\pi_n(n)=i_1$ . This means that the digit $1$ of column $n$ appears at row $i_1$. If $i_1=0$
then column $n$ is a zero column.
\item Denote by $O_n$, ($O_{n-1}$) the set of zero rows of $\pi_n$,
($\pi_{n-1}$) respectively.
\item Let $\pi_k$ be a partial permutation of order $k$. For each $1 \leq i \leq k$, the number of
zero columns of $\pi_k$, until column $i$
(including $i$ itself in case it is a zero column) is denoted by $o_k(i)$.

\item Let $$\Delta_n=inv(\tilde{\pi}_n)+exc(\tilde{\pi}_n)-(inv(\tilde{\pi}_{n-1})+exc(\tilde{\pi}_{n-1})).$$

\item If $\pi$ is a partial permutation on $n$ elements then we denote $$sup(\pi)=|\{1 \leq i \leq n \mid \pi(i)\neq 0\}\}|$$
so $\tilde{\pi} \in S_{\sup(\pi)}$.

\item For each partial involution $\pi$ of order $k$ and $1 \leq i \leq k$ such that $\pi(i) \neq 0$ one has:
$\pi(i)=j$ if and only if $\tilde{\pi}(i-o_k(i))=j-o_k(j)$.

\item For a partial involution $\pi$
$$exc(\tilde\pi)=\frac{sup(\pi)-fix(\pi)}{2}$$ where $fix(\pi)=|\{1 \leq i
\leq sup(\pi) \mid \pi(i)=i\}|$.

\end{itemize}

 It is sufficient to prove that
$$\mathfrak{D}(\pi_n)-\mathfrak{D}(\pi_{n-1})=\frac{\Delta_n}{2} +
\sumlim_{i \in O_n}(n+1-i)- \sumlim_{i \in O_{n-1}}(n-i)$$

We calculate first the L.H.S.

Recall that
$$\mathfrak{D}(\pi_n)-\mathfrak{D}(\pi_{n-1})=\{i \mid 1 \leq i \leq
n, r_{i_n}=r_{i-1,n-1}\}.$$

 Let $R(\pi_n)=(r_{k,l})_{1 \leq k,l \leq n}$ be the rank control
matrix of $\pi_n$ and let $i \in \{1,\dots ,n-1\}$. If $i \notin
O_{n-1}$ then $r_{i,n} \neq r_{i-1,n-1}$.  If $ i \in O_{n-1} -
O_{n}$
 then $\pi_n(n)=i$ and we have again $r_{i,n} \neq r_{i-1,n-1}$. If
 $i \in O_{n-1} \cap O_n$ then $r_{i,n}=r_{i-1,n-1}$ if and only if the
 digit $1$ of column $n$ of $\pi_n$ appears after the
row  $i$ or if is does not appear at all, i.e.
 $\pi_n(n)>i$ or $\pi_n(n)=0$. Thus

 $$
\mathfrak{D}(\pi_n)-\mathfrak{D}(\pi_{n-1})=
\left\{
\begin{array}{cc}
|\{i \mid \pi_n(n)>i\} \cap O_n | =o_n(i_1)  & {i_1 \neq 0} \\
  |O_n| & {i_1=0}
\end{array}
\right.
$$

Before calculating the R.H.S., note that if $\pi \in S_k$ is an involution then $$exc(\pi)=\frac{k-fix(\pi)}{2}$$
where $fix(\pi)=|\{1 \leq i \leq k \mid \pi(i)=i\}|$.
 We distinguish between three cases according to the value of $i_1=\pi_n(n)$:

\begin{enumerate}

\item $i_1=0$

In this case we have $\tilde{\pi}_{n}=\tilde{\pi}_{n-1}$ so
$\Delta_n=0$. We also have $|O_n|=|O_{n-1}|+1$ so that the R.H.S is
just $|O_n|$ as required.

\item $0<i_1 < n$.  Note first that in this case $fix(\tilde{\pi}_{n-1})=fix(\tilde{\pi}_{n})$. (Indeed, let
$k \in \{n-1,n\}$ and let $i \in \{1,\dots, n\}$ be such that $\pi_k(i)=j \neq 0$. Then  $i-o_k(i)$ is a fixed point of
$\tilde{\pi}_k$ if and only if $i-o_k(\i)=j-o_k(j)$. If (without loss of generality) $i<j$ then the number of zero columns between
$i$ and $j$ is equal to the total number of columns between $i$ and $j$ in $\pi_k$ which implies that $j=0$, a contradiction.
Thus we must have $i=j$ so $i-o_k(i)$ is a fixed point of $\tilde{\pi}_k$ if and only if $i$ is a fixed point of $\pi_k$. The only
difference between $\pi_{n}$ is at $i_1$ which can't be a fixed point, hence $fix(\tilde{\pi}_{n-1})=fix(\tilde{\pi}_n)$).

It is easy to see that $sup(\pi_{n})=sup(\pi_{n-1})+2$ and thus $exc(\pi_n)-exc(\pi_{n-1})=1$.

We turn now to the calculation of
$inv(\tilde{\pi}_n)-inv(\tilde{\pi}_{n-1})$.
 When we pass from $\tilde{\pi}_{n-1}$ to
$\tilde{\pi}_{n}$, we put $1$ in places $(n-o_n(n),i_1-o_n(i_1))$
and $(i_1-o_n(i_1),n-o_n(n))$. The columns after the column
$i_1-o_n(i_1)$ contribute
$n-o_n(n)-(i_1-o_n(i_1))=n-o_n(n)-i_1+o_n(i_1)$ inversions while the
rows after row $i_1-o_n(i_1)$ contribute $n-o_n(n)-(i_1-o_n(i_1))$
inversions. Since the $1$ in place $(n-o_n(n),i_1-o_n(i_1))$ was
counted twice, we have:

$$inv(\tilde{\pi}_n)-inv(\tilde{\pi}_{n-1})=2(n-o_n(n)-i_1+o_n(i_1))-1.$$

Now, we have:
\begin{align*}
\frac{\Delta_n}{2}&=\frac{exc(\tilde{\pi}_n)-exc(\tilde{\pi}_{n-1})+inv(\pi_n)-inv(\pi_{n-1})}{2}=\\
&=\frac{1+2\left(n-o_n(n)-i_1+o_n(i_1)\right)-1}{2}=\\
&=n-o_n(n)-i_1+o_n(i_1)=n-|O_n|-i_1+o_n(i_1)\,.
\end{align*}
Now, since $0< i_1 <n$, we have  $|O_n|=|I_{n-1}|$, $\sumlim_{i \in
I_{n-1}}{i}-\sumlim_{i \in O_n}{i}=i_1$ and we are done.

\item $i_1=n$:  The following facts are easy to verify in this case:
$$inv(\tilde{\pi}_{n})-inv(\tilde{\pi}_{n-1})=exc(\tilde{\pi}_{n})-exc(\tilde{\pi}_{n-1}),$$
$$|O_n|=|O_{n-1}|=|o_n(n)|=|o_n(i_1)|,$$ and from here we easily get
our result.

\end{enumerate}

\end{proof}

\begin{exa}\label{ex3}
$$
\pi=\left[\begin{matrix}
0 &0 &0 &1 &0 &0\\
0 &0 &0 &0 &1 &0\\
0 &0 &0 &0 &0 &0\\
1 &0 &0 &0 &0 &0\\
0 &1 &0 &0 &0 &0\\
0 &0 &0 &0 &0 &0
\end{matrix}\right]\,\,\,
R(\pi)=\left(\begin{matrix}
0 &0 &0 &1 &1 &1\\
0 &0 &0 &1 &2 &2\\
0 &0 &0 &1 &2 &2\\
1 &1 &1 &2 &3 &3\\
1 &2 &2 &3 &4 &4\\
1 &2 &2 &3 &4 &4
\end{matrix}\right)
$$
The diagonals with added zeros are $(0,0,0,0,2,4,4)$,
$(0,0,0,1,3,4)$, $(0,0,1,2,3)$, $(0,1,2,2)$, $(0,1,2)$ and
$(0,1)$. Thus, $\mathfrak D(\pi)=8$.

From the other hand, here we have $n=6$,
$$
\tilde{\pi}=\left[\begin{matrix}
0 &0 &1 &0 \\
0 &0 &0 &1 \\
1 &0 &0 &0 \\
0 &1 &0 &0
\end{matrix}\right]=(3412)\in
S_4\,\,,\,\,\textrm{and zero rows
of}\,\,\pi\,\,\textrm{are}\,\,\,\left\{i_1,i_2\right\}=\{3,6\}\,.
$$
So,
$$
\mathfrak
D(\pi)=\frac{exc(\tilde{\pi})+inv(\tilde{\pi})}{2}+\sum_{t=1}^{2}(6+1-i_t)=
\frac{2+4}{2}+4+1=8\,.
$$

\bco\label{incittigrfr}
The Bruhat poset of regular (not partial) involutions of $S_n$ is a graded poset with the rank function given by the formula
 $$\mathfrak
D(\sigma)=\frac{exc(\sigma)+inv(\sigma)}{2}\,\,,$$
where $\sigma\in Invol(S_n)$.
\eco
\begin{proof}
The fact that the Bruhat poset of involutions of $S_n$ is graded follows from the fact that this poset is an interval (more precisely, a reversed interval) in the poset of partial involutions (since the $(n,n)$-th entry in the rank-control matrix of any regular involution is $n$ because the matrix of $\sigma\in Invol(S_n)$ is invertible) and the fact that the poset of partial involutions is graded (Proposition~\ref{rankdimrenner}). The formula for the rank function follows from Theorem~\ref{genincitti} since the matrix of $\sigma\in Invol(S_n)$ is invertible and doesn't have zero rows or columns.
This is exactly the rank function of the Bruhat poset of the involutions of $S_n$ given by Incitti in~\cite{I}.
\end{proof}

\section{The Non-symmetric Case}
Consider the the double cosets  $\mathcal B_\pi=\left\{B_1^t\pi B_2\right\}$ where $\pi\in R_n$ (i.e. $\pi$ is a partial permutation and $B_1,B_2\in\mathbb B_n(\mathbb C)$. (Considering only $\pi\in S_n$ we get a version of the Bruhat decomposition for $GL_n(\mathbb C)$). Similarly to the definition of $\mathfrak D$ (Definition~\ref{eqcount}) we can define a parameter $\mathfrak E(\pi)$ for a not necessarily symmetric $\pi$ which counts all the equalities in all the diagonals of the rank-control matrix of $\pi$, not only in its upper triangle as we did for the symmetric case. Then we have the following formula for the dimension:
$$\dim\overline{\mathcal B_\pi}=n^2-\mathfrak E(\pi)\,.$$
Comparing this formula with the formula of Theorem~\ref{PosetRankFunction} we see that here we have $n^2$ instead of $\frac{n^2+n}{2}$ because the dimension of the variety of all $n\times n$ matrices is $n^2$. The proof is same as he proof of Theorem~\ref{PosetRankFunction} with obvious changes. Another case where the similar to $\mathfrak D$ parameter is used in the similar way is discussed in~\cite{C}, where the poset of Borel congruence classes of anti-symmetric
matrices is presented. In~\cite{C} it is shown that congruence B-orbits of anti-symmetric matrices can be indexed by involutions of the symmetric group and the parameter which is analogous to $\mathfrak D$ (denoted as $\mathfrak A$ in~\cite{C}) to counts the equalities in the diagonals of the "strict" upper triangle of the rank-control matrix (i.e the upper triangle without the main diagonal). There is the a formula (Theorem~5.6 of~\cite{C}) for the dimension of the closure of the congruence B-orbit $\mathcal A_\pi$ which corresponds to the involution $\pi\in S_n$:
$$
\dim\,\overline{\mathcal A_\pi}=\frac{n^2-n}{2}-\mathfrak A(\pi).
$$
The number $\frac{n^2-n}{2}$ is the dimension of the variety of all $n\times n$ anti-symmetric matrices.

So, the counting of the equalities in the diagonals of the rank-control matrix is a unified way to find the rank function of the (generalized) Bruhat poset, both in the general and in the symmetric case and also for the poset of congruence B-orbits of anti-symmetric matrices, i.e. the same must be done to compute the rank function of the Bruhat poset all (partial) permutations and the Bruhat poset (partial) involutions. The difference between these three cases is very natural: in the general case we consider all the diagonals of the rank-control matrix, while in the symmetric and anti-symmetric case we consider only the upper triangle (with or without the main diagonal respectively), since a symmetric matrix is completely determined by its upper triangle,  and an anti-symmetric matrix has zeros in its main diagonal and is completely determined by its "strict" upper triangle.

If we consider the double cosets $\left\{B_1\pi B_2\right\}$ (multiplying $\pi$ from both sides by upper-triangular matrices as L.~Renner does in his works~\cite{R}, \cite{R1} and~\cite{CR}) we also can define a parameter analogous to $\mathfrak E$ or $\mathfrak D$, but in this case we must go along the diagonals from south-west to the north-east (from the lower-left corner to the upper-right corner) of the rank-control matrix. In this case case the number of such equalities is equal to the dimension. The proof also is similar to the proof of Theorem~\ref{PosetRankFunction}.
\end{exa}

\noindent
{\bf Acknowledgements.} It is the pleasure for both authors  to thank Dr. Anna Melnikov whose work~\cite{M} gave the starting point for this our paper. Also we are grateful to Prof. Ron M. Adin for many very helpful discussions. We would like to express the special gratitude to Prof. Lex Renner for providing us the very useful information about the Bruhat poset and answering our numerous questions. Also we are grateful to Prof. Yuval Roichman and to Prof. Uzi Vishne for helpful discussions. The second author would like to thank the Faculty of Mathematics, Technion, Israel, and the Department of Mathematics, University of Geneva, Switzerland, for the hospitality and for the financial  support.


\begin{thebibliography}{LubW}
\bibitem
{BB} A.\ Bj\"orner and F.\ Brenti, {Combinatorics of Coxeter
groups},Springer GTM 231, 2004.
\bibitem
{CR} M.\ B.\ Can and L.\ E.\ Renner, {\it Bruhat-Chevalley order
on the rook monoid}, preprint (2008), available from http://arxiv.org/abs/0803.0491.
\bibitem{C}
Y. Cherniavsky, {\it Involutions of the Symmetric Group and Congruence B-orbits of Anti-Symmetric Matrices}, preprint, available from http://arxiv.org/abs/0910.4743.
\bibitem{I}
F.\ Incitti, {\it The Bruhat order on the involutions of the symmetric group},
Journal of Algebraic Combinatorics {\bf 20} (2004) 243-261.
\bibitem
{M} A.\ Melnikov, {\it Description of B-orbit closures of order 2
in upper-triangular matrices}, Transformation Groups, Vol. 11 No.
2, 2006, pp. 217-247.
\bibitem
{MS} E.\ Miller and B.\ Sturmfels, Combinatorial Commutative
Algebra, Springer GTM 227, 2005.
\bibitem
{R1} L.\ E.\ Renner, {\it Analogue of the Bruhat decomposition for
algebraic monoids}, J.\ Algebra {\bf 101} (1986), 303--338.
\bibitem
{R} L.\ E.\ Renner, Linear Algebraic Monoids, Springer, 2005.
\bibitem{S}
F.\ Szechtman, {\it Equivalence and Congruence of Matrices under the Action of Standard Parabolic Subgroups}, Electron. J.\ Linear Algebra {\bf 16} (2007), 325-333.
\end{thebibliography}
\end{document}